\documentclass[12pt]{article}

\usepackage{amsrefs}
\usepackage{amsmath,amssymb,amsthm}
\usepackage{enumerate,pxfonts}
\usepackage[all,cmtip]{xy}
\usepackage[applemac]{inputenc}
\usepackage[colorlinks=true, linkcolor=black, citecolor=black]{hyperref}

\newcommand \al{\alpha}
\newcommand \bs{\backslash}
\newcommand \C{{\mathbb C}}

\newcommand \CF{\mathcal{F}}

\newcommand \CO{\mathcal{O}}

\newcommand \eps{\varepsilon}

\newcommand \Ga{\Gamma}

\newcommand \ga{\gamma}
\newcommand \HH{{\mathbb H}}

\newcommand \InjRad{\operatorname{InjRad}}
\newcommand \la{\lambda}

\newcommand \N{{\mathbb N}}

\newcommand \Pl{\mathrm{Pl}}
\newcommand \PSL{\operatorname{PSL}}

\newcommand \R{{\mathbb R}}
\newcommand \SL{\operatorname{SL}}
\newcommand \SO{\operatorname{SO}}
\newcommand \sm{\smallsetminus}
\newcommand \supp{\operatorname{supp}}
\newcommand \tr{\operatorname{tr}}
\newcommand \triv{\mathrm{triv}}

\newcommand \vol{\operatorname{vol}}
\newcommand \what{\widehat}

\renewcommand \1{{\bf 1}}
\renewcommand \Im{\operatorname{Im}}
\renewcommand \Re{\operatorname{Re}}
\renewcommand \({\left(}
\renewcommand \){\right)}

\newcommand{\e}
[1]{\emph{#1}\index{#1}}

\newtheorem{theorem}{Theorem}[section]

\newtheorem{lemma}[theorem]{Lemma}

\theoremstyle{definition}
\newtheorem{definition}[theorem]{Definition}

\newtheorem{remark}[theorem]{Remark}

\begin{document}

\pagestyle{myheadings} \markright{Benjamini-Schramm and zeta}

\title{Benjamini-Schramm convergence and zeta functions}
\author{Anton Deitmar}
\date{}
\maketitle

{\bf Abstract:}
In this  note we show the equivalence of Benjamini-Schramm convergence and convergence of zeta functions for compact hyperbolic surfaces.

$$ $$

\tableofcontents

\newpage
\section*{Introduction}

Benjamini-Schramm convergence of metric probability spaces $X_n$ to a pointed metric space $(X,p)$ means that for every radius $R>0$ the likelihood of a point $x$ having the ball $B_R(x)$ of radius $R$ isometric with $B_R(p)$ tends to one, i.e.,
$$
P_n\(\big\{x\in X_n: B_R(x)\cong B_R(p)\big\}\)\longrightarrow 1,\qquad n\to\infty.
$$
For hyperbolic surfaces, Atle Selberg  introduced in \cite{Selberg} a geometric zeta function  which counts closed geodesics.
Ihara \cite{Ihara2by2} established a $p$-adic analog of this, which later was generalised to arbitrary graphs by Hashimoto and Hori \cite{Hashi}.
Lenz, Pogorzelski and Schmidt proved recently \cite{Lenz}, that a sequence of graphs of bounded valency is Benjamini-Schramm convergent to an infinite tree if and only if its Ihara zeta functions converge to the trivial one.

In the present paper the same assertion is established for compact hyperbolic surfaces.
The bounded valency condition is here replaced by a lower bound on the injectivity radius.
The proof of the main theorem uses the trace formula to transfer the statement to a spectral theoretic context where then a growth estimate on Laplace eigenvalues \cite{HKP} plays a crucial role.
In the first section we introduce the most important notions and collect material from our previous paper \cite{BSSpectral}. In the second section we state and prove the main theorem and in the third section we collect some further questions and projects which might come out of this paper.

\section{Plancherel and Benjamini-Schramm sequences}\label{S1}
Let $G=\PSL_2(\R)=\SL_2(\R)/\pm 1$. Then $K=\SO(2)/\pm 1$ is a maximal compact subgroup of $G$.
The group $G$ acts on the upper half plane $\HH=\{ z\in\C:\Im(z)>0\}$ via linear fractionals and this action induces an identification of  $G$ with the group of orientation-preserving isometries of the two-dimensional hyperbolic space.
We normalize the Haar measure on $K$ to have volume 1. Next we normalize the Haar measure on $G$ such that it induces the usual $\frac{dxdy}{y^2}$ on the upper half plane $\HH\cong G/K$.

For a cocompact lattice $\Ga\subset G$ the unitary representation of $G$, given by right translation on $L^2(\Ga\bs G)$ decomposes as a direct sum of irreducibles
$$
L^2(\Ga\bs G)\cong\bigoplus_{\pi\in\what G}N_\Ga(\pi)\pi.
$$
The multiplicities $N_\Ga(\pi)$ are finite and are zero outside a countable subset of the unitary dual $\what G$.

\begin{definition}
We  say that the measure on $\what G$ given by
$$
\mu_\Ga=\sum_{\pi\in\what G}N_\Ga(\pi)\,\delta_\pi
$$
is the \e{spectral measure} attached to $\Ga$.
\end{definition}

\begin{definition}
Let $(\Ga_n)$ be a sequence of cocompact lattices in $G$.
We say that the sequence is a \e{Plancherel sequence}, if for every $f\in C_c^\infty(G)$ we have
$$
\frac1{\vol(\Ga_n\bs G)}\int_{\what G}\hat f(\pi)\,d\mu_{\Ga_n}(\pi)\ \longrightarrow f(e)
$$
as $n\to\infty$, where $\hat f(\pi)=\tr\pi(f)$.

By the Plancherel Theorem we have  $f(e)=\int_{\what G}\hat f(\pi)\,d\mu_\Pl(\pi)$, where $\mu_\Pl$ is the Plancherel measure on $\what G$, so that the sequence $(\Ga_n)$ is Plancherel if and only if in the dual space of $C_c^\infty(G)$ one has weak-*-convergence
$$
\frac1{\vol(\Ga_n\bs G)}\mu_{\Ga_n}\ \longrightarrow\ \mu_\Pl.
$$
\end{definition}

\begin{remark}
\begin{enumerate}[\rm (a)]
\item If a sequence $(\Ga_n)$ of  lattices is a Plancherel sequence, then
$$
\frac1{\vol(G/\Ga_n)}\mu_{\Ga_n}(U)\ \longrightarrow\ \mu_\Pl(U)
$$
for every relatively compact open set $U\subset\what G$, whose boundary has Plancherel measure zero.
This follows from the density principle of Sauvageot \cite{Sauvageot}.
\item
If $(\Ga_n)$ is a Plancherel sequence, then 
$$
\vol(\Ga_n\bs G)\longrightarrow 0
$$
as $n\to\infty$.
This follows from the fact that the spectral measure of each $\Ga_n$ is discrete and the Plancherel measure is not.
\end{enumerate}
\end{remark}

\begin{definition}
A sequence $(\Ga_n)$ of lattices is called \e{uniformly discrete}, if there exists a unit-neighborhood $U\subset G$ such that $x^{-1}\Ga_nx\cap U=\{1\}$ holds for every $x\in G$.
\end{definition}

\begin{definition}
In \cite{7Samurais} the sequence of spaces $\Ga_n\bs \HH$ is said to be \e{Benjamini-Schramm convergent} or \e{BS-convergent} to $\HH$ if for every $R>0$
$$
P_{\Ga_n\bs \HH}\(\big\{x\in\Ga_n\bs X: \InjRad(x)\le R\big\}\)
$$
tends to zero, where $\InjRad(x)$ is the injectivity radius at the point $x$.

In this case we say that the sequence $(\Ga_n)$ is a \e{BS-sequence}.
\end{definition}

\begin{remark}
In \cite{BSSpectral} it is shown that 
\begin{align*}
&(\Ga_n)\text{ is BS and uniformly discrete}\\
&\Rightarrow (\Ga_n)\text{ Plancherel}\\
&\Rightarrow (\Ga_n)\text{ is BS}.
\end{align*}
\end{remark}

\begin{lemma}
Let $(\Ga_n)$ be a Plancherel sequence and let $(\Sigma_n)$ be a sequence of sublattices $\Sigma_n\subset\Ga_n$.
Then $(\Sigma_n)$  again is a Plancherel sequence.
\end{lemma}

\begin{proof}
In \cite{BSSpectral}, Proposition 2.9 it is shown that  a sequence of cocompact lattices $(\Ga_n)$ is Plancherel if and only if for every compact set $C\subset G$ the sequence
$$
\frac1{\vol(\Ga_n\bs G)}\int_{\Ga_n\bs G}\#\big(x^{-1}\Ga_n^*x\cap C\big)\,dx
$$
tends to zero as $n\to\infty$, where $\Ga_n^*=\Ga_n\sm\{1\}$.
Now let $(\Sigma_n)$ be a sequence as in the lemma.
Then, being a lattice, each $\Sigma_n$ has finite index in $\Ga_n$.
Fix a fundamental domain $\CF_n$ of $\Ga_n\bs G$.
Then
\begin{align*}
&\frac1{\vol(\Sigma_n\bs G)}\int_{\Sigma_n\bs G}
\#\big(x^{-1}\Sigma_n^*x\cap C\big)\,dx\\
&=\frac1{\vol(\Sigma_n\bs G)}\sum_{\ga\in\Sigma_n\bs\Ga_n}\int_{\CF_n}
\#\big(x^{-1}\ga^{-1}\Sigma_n^*\ga x\cap C\big)\,dx\\
&\le\frac{[\Ga_n:\Sigma_n]}{\vol(\Sigma_n\bs G)}\int_{\CF_n}
\#\big(x^{-1}\Ga_n^* x\cap C\big)\,dx\\
&=\frac{1}{\vol(\Ga_n\bs G)}\int_{\Ga_n\bs G}
\#\big(x^{-1}\Ga_n^* x\cap C\big)\,dx.
\end{align*}
As the latter tends to zero, so does the former.
\end{proof}

\section{The Selberg zeta function}

Let $(\Ga_n)$ be a Plancherel-sequence in $G$.
For simplicity, we shall assume that each $\Ga_n$ is torsion-free, which can easily be arranged as every lattice $\Ga$ contains a torsion-free  sublattice.

\begin{definition}
The Selberg zeta function for $\Ga_n$ is defined for $s\in\C$ with $\Re(s)>1$ as
$$
Z_n(s)= \prod_{\ga}\prod_{k\ge 0} \( 1-e^{-(s+k)l(\ga)}\),
$$
where the first product runs over all primitive hyperbolic conjugacy 
classes in $\Ga_n$ (see \cite{HA2}, Section 11.6). The product converges for $\Re(s)>1$ and the so defined function extends holomorphically to all of $\C$.
\end{definition}

\begin{theorem}
Let $(\Ga_n)$ be a sequence of torsion-free cocompact lattices in $G$.
\begin{enumerate}[\rm (a)]
\item If the sequence $(\Ga_n)$ is uniformly dscrete and  Plancherel, then  
$$
\frac1{\vol(\Ga_n\bs G)}\frac{Z_n'}{Z_n}(s)
$$ 
converges to zero in the set $\{\Re(s)>1\}$.
\item If 
$$
\frac1{\vol(\Ga_n\bs G)}\frac{Z_n'}{Z_n}(s)
$$ 
converges to zero in the set $\{\Re(s)>1\}$, then the sequence is Plancherel.
\end{enumerate}
In either case, the convergence of $\frac1{\vol(\Ga_n\bs G)}\frac{Z_n'}{Z_n}(s)$ is uniform on every set of the form $\{\Re(s)\ge\al\}$, for $\al>1$. 
\end{theorem}

\begin{proof}
For $\Re(s)>1$ we have
\begin{align*}
\frac{Z_n'}{Z_n}(s)=\sum_{k\ge 0}\sum_{[\ga]}\ell(\ga_0)e^{-(s+k)\ell(\ga)}=\sum_{[\ga]}\ell(\ga_0)\frac{e^{-s\ell(\ga)}}{e^{\ell(\ga)}-1},
\end{align*}
where the sum runs over all conjugacy classes $[\ga]\ne\{1\}$ in $\Ga_n$ and $\ga_0$ is the underlying primitive of $\ga$, i.e., $\ga=\ga_0^m$ for some $m\in\N$.
Now let $\Re(s)\ge\al>1$, then $\left|e^{-(s+k)\ell(\ga)}\right|=e^{-(\Re(s)+k)\ell(\ga)}\le e^{-(\al+k)\ell(\ga)}$ and so the addendum follows.

(a)
Now suppose that the sequence $(\Ga_n)$ is Plancherel.
By \cite{HA2}, Section 11, we have
\begin{align*}
L^2(\Ga_n\bs \HH)=L^2(\Ga_n\bs G)^K=\C\oplus\bigoplus_{j=1}^\infty \pi_{ir_j},
\end{align*}
where $\C$ stands for the one-dimensional space of constant functions and $\pi_{ir}$ is the induced representation (Principal or Complementary series) with $r\in \R\cup i\left(-0,\frac12\right)$.
Formally we set $r_0=\frac i2$.
By Section 11.6 of \cite{HA2} we then have 
\begin{align*}
\frac1s\frac{Z_n'}{Z_n}\(s+\frac12\)&=\frac{1}{b}\frac{Z_n'}{Z_n} \(b+\frac12\)-\frac{\vol(\Ga_n\bs G)}\pi
\sum_{n=0}^\infty\left( \frac{1}{s+\frac12+n} -\frac{1}{b+\frac12+n}\right) \\
& + 2\sum_{j=0}^\infty\frac{1}{s^2+r_j^2}-
\frac{1}{b^2+r_j^2}.
\end{align*}
Let $D_s$ denote the operator $D_s(\psi)(s)=-\frac{\partial}{\partial s}\(\frac1s\psi(s)\)$
We get
\begin{align*}
D_s\frac{Z_n'}{Z_n}\(s+\frac12\)
&=
4s\sum_{j=0}^\infty\frac1{(s^2+r_j^2)^2}
-\frac{\vol(\Ga_n\bs G)}\pi\sum_{n=0}^\infty \frac1{(s+\frac12+n)^2}
\end{align*}
and
\begin{align*}
D_s^2\frac{Z_n'}{Z_n}\(s+\frac12\)
&=
-\frac{\vol(\Ga_n\bs G)}\pi\(\frac1{s^2}\sum_{n=0}^\infty \frac1{(s+\frac12+n)^2}+2\sum_{n=0}^\infty\frac1{(s+\frac12+n)^3}  
\)\\
&\ \ \ +8s\sum_{j=0}^\infty\frac1{(s^2+r_j^2)^3}
\end{align*}
For $T\ge 1$ let $N_n(T)=\#\big\{ j:\big|\frac14+r_j^2<T\big\}$.
Recall that $\la_j=\frac14+r_j^2$ is the $j$-th Laplace eigenvalue.
To make present the dependence on $n$ we write $\la_j^{(n)}$.
The sequence $(\Ga_n)$ being uniformly discrete means that the injectivity radii of the manifolds $\Ga_n\bs\HH$ are bounded below. 
Therefore, by formula (1.2.5) of \cite{HKP}, there exists a constant $C>0$ such that for every $T\ge 1$ one has
$$
N(T)\le C\,\vol(\Ga_n\bs G)\, T.
$$
Let
$$
h_s(\la)=\frac1{\(s^2+\la-\frac14\)^3}.
$$
For $s>\frac12$ this function is positive on $[0,\infty)$.
It is continuous and monotonically decreasing.

Let $\what G_K$ denote the set of all 
$\pi\in\what G$ such that the representation space $V_\pi$ contains non-zero $K$-fixed vectors.
Then it is known \cite{Knapp},
$$
\what G_K=\left\{\pi_{ir}: r\in i\left(0,\frac12\right)\cup\R_{\ge 0}\right\}\cup\{\triv\}.
$$
The map $\phi:\what G_K\to [0,\infty)$,
given by
\begin{align*}
\phi(\triv)&=0,\\
\phi(\pi_{ir})&=\frac14+r^2
\end{align*}
is a homeomorphism.
We shall from now on identify $\what G_K$ with $[0,\infty)$.
If $I\subset [0,\infty)$ is a relatively open, bounded interval, then by \cite{Sauvageot}
we have
$$
\lim_{n\to\infty}\frac1{\vol(\Ga_n\bs G)}\mu_n(\1_I)
=\mu_{\Pl}(\1_I),
$$
where $\1_I$ is the indicator function of $I$.
By linearity this extends to linear combinations of functions of the form $\1_I$.
There exists a sequence $(L_k)_{k\in\N}$ of such linear combinations such that $0\le L_k\nearrow h_s$ outside a countable set $S$, which is of Plancherel measure zero and can also be chosen to be of $\mu_{\Ga_n}$ measure zero for all $n$ and have empty intersection with $\N$.
We can also choose the $L_k$ so that for each $T\in\N_0$ we have
$$
\Phi_k(T)=\sup_{x\in[T,T+1)\sm S}h_s(x)-L_k(x)
$$
tends to zero for $k\to\infty$.
For brevity, we write $\mu(f)$ instead of $\int_Xf\,d\mu$ where $\mu$ is a measure on $X$ and $f$ a function.
We also write $\mu_n=\mu_{\Ga_n}$.
Now, since $N(0)=0$,
\begin{align*}
0\le\mu_n(h_s-L_k)&=\sum_jh_s(\la_j)-L_k(\la_j)\\
&\le\sum_{T=0}^\infty \Phi_k(T)\big(N(T+1)-N(T)\big)\\
&\le C\vol(\Ga_n\bs G)\sum_{T=0}^\infty\Phi_k(T)(2T+1)
\end{align*}
Since $\Phi_k(T)\le 2h_s(T)$, we can apply
dominated convergence, to get that this sum tends to zero for $k\to\infty$.
So let $\eps>0$. Then there is $k_0\in\N$ such that for all $k\ge k_0$ we have
$0\le\mu_n(h_s)-\mu_n(L_k)<\vol(\Ga_n\bs G)\eps/3$ holds for all $n\in\N$ and that $|\mu_\Pl(L_k)-\mu_\Pl(h_s)|<\eps/3$.

Fix some $k\ge k_0$. Then there exists $n_0\in\N$ such that for all $n\ge n_0$ one has 
$$
\left|\frac1{\vol(\Ga_n\bs G)}\mu_n(L_k)-\mu_\Pl(L_k)\right|<\eps/3.
$$
And so
\begin{align*}
\left|\frac1{\vol(\Ga_n\bs G)}\mu_n(h_s)-\mu_\Pl(h_s)\right|
&\le \frac1{\vol(\Ga_n\bs G)}|\mu_n(h_s)-\mu_n(L_k)|\\
&+\left|\frac1{\vol(\Ga_n\bs G)}\mu_n(L_k)-\mu_\Pl(L_k)\right|\\
&+|\mu_\Pl(L_k)-\mu_\Pl(h_s)|<\frac\eps 3+\frac\eps 3+\frac\eps 3=\eps.
\end{align*}
This means that
$D_s^2\frac{Z_n'}{Z_n}(s)$ converges to zero for $\Re(s)>1$.
The following  lemma proves `only if' direction of  the theorem.

\begin{lemma}
For $n,k\in\N$ let $a_{n,k},b_{n,k}>0$ be real numbers. Suppose that  $L_n(s)=\sum_{k=1}^\infty a_{n,k}e^{-sb_{n,k}}$ converges for $\Re(s)>1$ and that $D_sL_n(s)$ tends to zero as $n\to\infty$.
Then $L_n(s)$ also tends to zero as $n\to\infty$ for every $s$ with $\Re(s)>1$.
\end{lemma}

\begin{proof}
As the sum $L_n(s)$ converges locally uniformly, by the Theorem of Weierstraß, we can differentiate under the sum to get for $s>1$ that
$$
s^2D_sL_n(s)=\sum_{k=1}^\infty a_{n,k}(sb_{n,k}+1)e^{-sb_{n,k}}\ge \sum_{k=1}^\infty a_{n,k}e^{-sb_{n,k}}=L_n(s)\ge 0.
$$ 
Now if the former tends to zero, then so will the latter.
\end{proof}

(b)
For the converse direction assume convergence of $\frac1{\vol(\Ga_n\bs G)}\frac{Z_n'}{Z_n}(s)$ to zero and  let $f\in C^\infty_c(G)$.
As $f$ has compact support, there exists $c>0$ such that for every $n\in\N$ and every $\ga\in\Ga_n\sm\{1\}$ with $l(\ga)>c$ the conjugation orbit
$\{x\ga x^{-1}:x\in G\}$ has empty intersection with $\supp(f)$.
Hence for such $\ga$ we have $\CO_\ga(f)=0$.
Here $\CO_\ga(f)=\int_{G/G_\ga}f(x\ga x^{-1})\,dx$
is the orbital integral and $G_\ga$ is the centralizer of $\ga$ in $G$.
As $f$ is bounded and has compact support, there exists $M>0$ such that $\ell(\ga_0)|\CO_g(f)|\le M$ for all hyperbolic $g\in G$.
Note that  any cocompact lattice only contains hyperbolic elements besides the trivial one.
By the trace formula \cite{HA2}, it follows that for given $s>1$ we have
\begin{align*}
\left|\frac1{\vol(\Ga_n\bs G)}\mu_n(f)-f(e)\right|
&\le \frac M{\vol(\Ga_n\bs G)}\sum_{[\ga]\ne [e], \ell(\ga)\le c}\ell(\ga_0)\\
&\le \frac M{\vol(\Ga_n\bs G)}\frac{e^c-1}{e^{-sc}}\sum_{[\ga]\ne [e], \ell(\ga)\le c}\ell(\ga_0)\frac{e^{-s\ell(\ga)}}{e^{\ell(\ga)}-1}\\
&= \frac M{\vol(\Ga_n\bs G)}\frac{e^c-1}{e^{-sc}}\frac{Z_n'}{Z_n}(s)\longrightarrow 0
\end{align*}
as $n\to\infty$. The theorem is proven.
\end{proof}

\section{Open questions and further projects}
\subsection*{Uniform discreteness}
One assertion of the main theorem was proven under the condition of uniform discreteness, or, equivalently, a lower bound on the injectivity radius.
It is not clear whether that condition is necessary.
It seems impossible to eliminate the injectivity radius from the eigenvalue estimates, as Theorem 8.1.2 in \cite{Buser} shows.
According to this theorem, for fixed genus $g$ (and therefore fixed volume $\vol(\Ga\bs G)$), for every $\eps>0$ there exist groups $\Ga$ of genus $g$ and  $N(1+\eps)$ arbitrarily large.
Therefore the only option seems to lie in 
an analysis of the Teichmüller space along the lines of \cite{Monk} and the references therein
(Although possible critical cases have been excluded in that paper).

\subsection*{General rank one groups}
The present proof uses eigenvalue estimates for the Laplacian. For general rank one Lie groups like $\mathrm{SO}(n,1)$ the Selberg zeta function  is described by the spectrum of generalized Laplacians \cite{BGV} on certain homogeneous vector bundles.
An extension of the present results would therefore require an extension of the eigenvalue estimates to these bundles.
One possible path might be the use of the ``group Laplacian'' instead, which would provide a much weaker estimate, as the dimension increases, however, it might be sufficient for the task at hand.

\subsection*{Higher rank}
For higher rank groups the Selberg zeta function needs replacing by corresponding higher rank zeta functions as in \cite{HR}.
As this zeta function only collects closed geodesics which lie in an open Weyl chamber, it might be necessary to consider several zeta functions, one for each conjugacy class of non-compact Cartan subgroups.
On the other hand, a simplification may arise by only considering the restriction of these several variable zeta functions to generic lines. 

\subsection*{p-adic groups}
For $p$-adic groups the symmetric space is replaced with the Bruhat-Tits building.
The rank one case (i.e. the case of graphs) has in great generality been dealt with affirmatively in \cite{Lenz}.
The higher rank case will rely on the several variable zeta functions defined in \cite{padicHR} and otherwise face the same difficulties as in the Lie-group situation except for the fact, that small radii of injectivity play nor role here.
 
\subsection*{Locally compact groups}
This last and most general case is highly speculative.
Is it possible to give a zeta function for any uniform lattice $\Ga$ in an arbitrary locally compact group $G$ which reflects the global geometry well enough to detect Benjamini-Schramm convergence as formulated in \cite{BSSpectral}?
There is a possible top-down and a bottom-up approach to this problem.
The top-down approach uses the data given in the trace formula to define a new type of zeta function such that the spectral side o the trace formula yields analytic continuation.
The bottom-up approach uses the known cases and the structure theory of locally compact groups given for instance in \cite{Tao}.

\subsection*{Non-cocompact lattices}
For arithmetic congruence groups, the adelic trace formula can be used to show that certain sequences of arithmetic groups are BS, see \cites{Raimbault,Matz}.
An open problem raised in these papers, is the question if any sequence $(\Ga_n)$ of congruence subgroups in a given linear algebraic group $G$ is already BS if the covolumes tend to infinity. 
A similar statement is known to be wrong without the congruence property.

The connection to Selberg-type zeta functions is more subtle in the noncompact situation, as the trace formula does not provide a direct link between geometric
spectral data.

{\small Mathematisches Institut\\
Auf der Morgenstelle 10\\
72076 T\"ubingen\\
Germany\\
\tt deitmar@uni-tuebingen.de}

\today

\end{document}